\documentclass[11pt,twoside,leqno,titlepage]{article}
\usepackage{amsmath}
\usepackage{amsthm}
\usepackage{amsfonts}
\usepackage{amssymb}
\usepackage{vmargin}
\input xy
\xyoption {all}

\numberwithin{equation}{section}

\newtheorem{thm}[equation]{Theorem}

\newtheorem{defn}{Definition}
\newtheorem{rem}{Remark}
\newcommand{\R}{\mathbb{R}}

\newcommand{\ent}[2]{\mathop{\mathrm{Ent}_{#1}}\left(#2\right)}
\renewcommand{\epsilon}{\varepsilon}

\def\av_#1{\mathchoice%
        {\mathop{\kern 0.2em\vrule width 0.6em height 0.69678ex depth -0.58065ex
                \kern -0.7em \intop}\nolimits_{\kern -0.4em#1}}%
        {\mathop{\kern 0.1em\vrule width 0.5em height 0.69678ex depth -0.60387ex
                \kern -0.6em \intop}\nolimits_{#1}}%
        {\mathop{\kern 0.1em\vrule width 0.5em height 0.69678ex depth -0.60387ex
                \kern -0.6em \intop}\nolimits_{#1}}%
        {\mathop{\kern 0.1em\vrule width 0.5em height 0.69678ex depth -0.60387ex
                \kern -0.6em \intop}\nolimits_{#1}}}

\title{Functional inequalities and Hamilton-Jacobi Equations in Geodesic Spaces\thanks{
 \hfill \break {\it Keywords\,}: Logarithmic-Sobolev inequalites, Talagrand inequalites,
Hamilton-Jacobi semigroup, Poincar\'e inequalities, geodesic metric space, metric-measure space.}}
\author{Zolt\'an M. Balogh \thanks{Corresponding author; Z. M. B. supported by the Swiss
Nationalfond, EC Project GALA: "Sub-Riemannian geometric analysis in Lie groups", and ERC Project HCAA, "Harmonic and complex analysis and applications"} \\ Institute of Mathematics \\ University of Bern \\Sidlerstrasse 5 \\ 3012 Bern, Switzerland \\ {\tt zoltan.balogh@math.unibe.ch} 
\and Alexandre Engoulatov \thanks{A. E.  supported by the EC Project GALA:"Sub-Riemannian geometric analysis in Lie groups"} \\ Institute of Mathematics \\ University of Bern \\
Sidlerstrasse 5 \\ 3012 Bern, Switzerland \\ {\tt alexander.engulatov@math.unibe.ch} 
\and Lars Hunziker \\ Department of Mathematics \\ University of Technology, Sydney \\
PO Box 123 Broadway \\ NSW 2007 Australia \\ {\tt lars.hunziker@uts.edu.au} 
 \and Outi Elina Maasalo \thanks{O. E. M.   supported by the EC Project GALA:"Sub-Riemannian geometric analysis in Lie groups"} Institute of Mathematics \\ University of Bern \\
Sidlerstrasse 5 \\ 3012 Bern, Switzerland \\ {\tt maasalo@math.unibe.ch}}
\date{\today}
\begin{document}
\maketitle
\begin{abstract}
We study the connection between the $p$--Talagrand inequality and the $q$--logarithmic Sololev inequality for conjugate exponents $p\geq 2$, $q\leq 2$ in proper geodesic metric spaces. By means of a general Hamilton--Jacobi semigroup we prove that these are equivalent, and moreover equivalent to the hypercontractivity of the Hamilton--Jacobi semigroup. Our results generalize those of Lott and Villani. They can be applied  to deduce the $p$-Talagrand inequality  in the sub-Riemannian setting of the Heisenberg group. 
\end{abstract}

\section{Introduction}

\noindent The main purpose of the present paper is to study
relations between functional inequalities on proper geodesic
metric measure spaces. More precisely, we prove that under some
additional assumption on the space, the $q$--logarithmic Sobolev
inequality and the  $p$--Talagrand inequality are equivalent for
the conjugate exponents $p\ge 2$ and $q\le 2$. This generalizes
the recent results of Lott and Villani, who considered similar
questions in the quadratic case when $p=q=2$; see \cite{LoVi2}. As
in \cite{LoVi2}, the Hamilton--Jacobi infimum convolution operator
plays a crucial role in our approach. This idea  goes back to the
work of Bobkov et al., \cite{BobGenLed}. They proved that in
Euclidean spaces a measure $\mu$ which is absolutely continuous
with respect to the Lebesgue measure satisfies the classical
logarithmic Sobolev inequality if and only if the Hamilton--Jacobi
semigroup associated to the quadratic infimum--convolution
operator is hypercontractive. Gentil and Malrieu generalized this to a broader class of logarithmic Sobolev inequalities; see \cite{GeMa}.

Lott and Villani applied the same strategy on a compact length
space $(X,d)$ equipped with a Borel probability measure $\mu$ to
prove the following. If the space supports a local Poincar\'e
inequality and the measure is doubling, then the quadratic
logarithmic Sobolev inequality implies the quadratic Talagrand
inequality with the same constant. In both proofs,
\cite{BobGenLed} and \cite{LoVi2}, it is crucial that the
infimum--convolution semigroup solves the Hamilton--Jacobi
equation associated to a radial Hamiltonian.

On the other hand, starting with a Talagrand inequality it is
possible to derive a logarithmic Sobolev inequality as a
consequence of the so called HWI inequality, which relates entropy
(H), Wasserstein distance (W) and Fisher information (I). However,
this requires an additional geometric assumption on the space. For
example, in the Riemannian setting it is sufficient to assume that
the reference measure $\mu$ satisfies the
Bakry--Emery~\cite{BakEm} curvature-dimension inequality
$CD(R,\infty)$ with the constant $R>-K$; see ~\cite{BobGenLed}. In the
more general setting of metric measure spaces we show that this is
guaranteed by the assumption that the entropy functional on the
Wasserstein space is weakly displacement convex. The notion of
weak displacement convexity is defined in the work of Lott and
Villani~\cite{LoVi}. See also \cite{St1} and \cite{St2} for questions related to the Ricci curvature in metric measure spaces.

To summarize our results we denote the  $q$--logarithmic Sobolev
inequality by $q{\text-LSI}$. We also introduce a notion of a
$p$--Talagrand inequality, $p{\text -T}$, where $p\ge 2$ and $q\le
2$ are conjugates so that $1/p+1/q=1$. We prove that
\begin{equation}
\xymatrix@1{
{\;\textrm{HC(}p\textrm{)}  \;}
\ar@{<=>}[rr]^-{\textrm{H-J}} &&
%\stackrel{\textrm{H-J}}{\Longleftarrow\!\Longrightarrow}
{\;\textit{q-LSI}\;}
\ar@/^/ @{=>}[rr]^-{\textrm{H-J}}  && { \; \textit{p-T}} \ar@/^/ @{=>}[ll]^-{\textrm{DConv}
}\;}.
\end{equation}
The left--hand side of the diagram represents the hypercontractivity
of the infimum--convolution semigroup associated to the exponent
$p$, {H-J} means that the implication is obtained via validity of
the Hamilton--Jacobi equation, and {DConv} stands for the weak
displacement convexity of the entropy functional.

The paper is organised as follows. In Section~\ref{defs} we list
some of the important properties of the infimum--convolution
semigroup. In Section~\ref{s-hc} we establish the equivalence on
the left-hand side of the above diagram, provided that the
Hamilton--Jacobi equation is satisfied. (It is the case e.g. when
the measure $\mu$ is doubling and supports a local Poincar\'e
inequality.)
%again assuming that the Hamilton-Jacobi equation is satisfied on $X$.
In Section~\ref{s-t} we consider the relation between the
$q$--logarithmic Sobolev inequality and the $p$--Talagrand
inequality. Again assuming that the Hamilton--Jacobi equation is
satisfied on $X$, we show that
%if the Hamilton-Jacobi equation is
%satisfied (e.g. if the measure $\mu$ is doubling and supports a
%local Poincar\'e inequality), then
the $q$--logarithmic Sobolev inequality implies the $p$--Talagrand
inequality. The converse implication holds under the assumption of
the weak displacement convexity of the entropy functional on the
Wasserstein space of probability measures on $X$. For the reader's
convenience Section~\ref{proofs} provides an account of the
infimum-convolution semigroup on proper length spaces. The final section 
is for remarks and further questions. We also indicate here 
an application of our results by using a recent result of Inglis and Papageorgiou 
\cite{InPa} on the logarithmic Sobolev inequality  in the sub-Riemannian setting of the Heisenberg group.

\section{Preliminaries, the Hamilton--Jacobi equation.}
\label{defs} Let $(X,d)$ be a metric space. We say that $d$ is a
length metric, if for all $x,y \in X$ we have
\begin{equation*}
d(x,y) = \inf \textrm{length}(\gamma),
\end{equation*}
where the infimum is taken over all paths that connect $x$ and
$y$. Notice, that if $X$ is proper, i.e. its closed and bounded
sets are compact, then the infimum is attained and the space is,
in fact, geodesic \cite{AmTi}.

We remind the reader that a Borel measure $\mu$ is doubling, if
the measure of any open ball is positive and finite, and if there
exists a constant $c_d\geq 1$ such that
\[
\mu(B(x,2r)) \leq c_d \mu(B(x,r))
\]
for all $x\in X$ and $r>0$.  Here $B(x,r)$ denotes an open ball of
radius $r$ centered in $x$.

If $f$ is a real--valued Lipschitz function on $X$, we write
\[
\mathrm{lip}f(x)=\liminf_{r\to 0}\sup_{d(x,y)<r}
\frac{|f(x)-f(y)|}{r}
\]
for every $x \in X$.

Let \mbox{$1\leq p<\infty$}. We say that $(X,d,\mu)$ satisfies a
local $(1,p)$--Poincar\'e inequality (see, for example,
\cite{HeKo}) if there exists $1\le L <\infty$ and $C>0$, such that
for all Lipschitz functions $f$ we have
\begin{equation}\label{locPoin}
\av_{B(x,r)} |f-f_{B(x,r)}|\, d\mu \leq Cr \bigg(\av_{B(x,Lr)}
(\mathrm{lip}f)^p\, d\mu\bigg)^{1/p}
\end{equation}
for all $x\in X$ and $r>0$. Here we wrote
$$
f_{B(x,r)} = \av_{B(x,r)} f\,d\mu = \frac 1{\mu(B(x,r))}
\int_{B(x,r)} f\, d\mu.
$$

We remind the reader that if $\mu$ is doubling and the metric
space is complete, the above definition coincides with the a
priori stronger definition involving upper gradients; see
\cite{HeKo} and \cite{Ke}.

\smallskip

Throughout the paper we assume that $d$ is a
length metric and $(X,d)$ is proper. Without further notice all
measures on $(X,d)$ will be Borel probability measures.  We will
later impose further assumptions on the space when they are
needed.

\subsection{Metric gradient and Hamilton--Jacobi equation in geodesic spaces.}

Consider a function $f\colon X\times \R_{+} \to \R$. We define the
so called metric gradient of $f$ with respect to the variable
$x\in X$ at a point $(x_0,t)  \in X\times \R_+$ as
\begin{equation*}
|\nabla f|(x_0,t):=\limsup_{x\rightarrow x_0}\frac{\left|f(x,t)-f(x_0,t)\right|}{d(x,x_0)}.
\end{equation*}
For an arbitrary function this could be infinite, but if $f$ is
Lipschitz continuous in the $x$ variable, the metric gradient $|\nabla
f|(x_0,t) $ is always finite. However, it turns out that for the
Hamilton--Jacobi equation in metric spaces one should consider a
slightly different notion of a gradient. Following the lines in
\cite{LoVi2}, we introduce the so called {\it metric subgradient}
of $f$ defined as
\begin{equation*}|\nabla^- f|(x_0,t):=\limsup_{x\rightarrow x_0}\frac{\left[f(x,t)-f(x_0,t)\right]_-}{d(x,x_0)}=\limsup_{x\rightarrow x_0}\frac{\left[f(x_0,t)-f(x,t)\right]_+}{d(x,x_0)},
\end{equation*}
where $a_+= \max (a,0)$ and $a_-= \max (-a,0)$. Notice,  that
\[
|\nabla^- f|(x_0,t)|\leq |\nabla f|(x_0,t),
\]
and $|\nabla^- f|(x_0,t)$ vanishes if $f(\cdot,t)$ has a local
minimum at $x_0$. In fact, the metric subgradient indicates that the
local variation of $f(\cdot,t)$ takes into account only values
less than $f(x_0,t)$.

In analogy to the Euclidean case (see, for example, Evans
\cite{Ev}) the initial-value problem for the Hamilton--Jacobi
equation in a geodesic space can be defined as
\begin{eqnarray}
\left\{ \begin{array}{ll}
\frac{\partial}{\partial t}u(x,t)+H\big(|\nabla^- u|(x,t)\big) &= 0 \phantom{g(x)}  \textrm{in $X\times\R_+$}\\
\phantom{H(|\nabla^- u|_t(x,t))t))}u(x,t) &= g(x)\phantom{0}
\textrm{on $X\times\{t=0\}$}. \label{HJEX}
\end{array} \right.\end{eqnarray}

Throughout the paper we assume that the initial data $g\colon
X\rightarrow\R$ is Lipschitz continuous and the function $H\colon
\R_+ \rightarrow \R_+$ is convex, superlinear and satisfies the
condition $H(0)=0$. Here $H$ is called the \emph{Hamiltonian}, and
in the Euclidean case a standard example for such a function is
$x\mapsto\frac{1}{\alpha}\left|x\right|^\alpha$ for a real $\alpha
1$.

The corresponding Hopf--Lax formula (or the infimum--convolution)
is defined by
\begin{equation}
Q_tg(x)=\inf_{y\in
X}\Big[tL\bigg(\frac{d(x,y)}{t}\bigg)+g(y)\Big], \label{HLX}
\end{equation}
where $L\colon\R_+\to \R_+$ is simply the one--dimensional Legendre transform of $H$ defined by
\begin{equation}\label{Legendre}
L(u)= \sup_{v\in\R_+}\{ uv-H(v)\} , \ u\in \R_+.
\end{equation}
Notice, that by standard results the one--dimensional Legendre
transformation  $L$ is increasing, convex, superlinear and
satisfies $L(0)=0$. Moreover,
\[H(w) = \max _{v\in \R_+}\{ w v
-L(v)\}.
\]

We remind the reader that in the Euclidean case the Hopf--Lax
formula provides a Lipschitz--continuous solution to the
Hamilton--Jacobi equation \cite{Ev}. This has been generalized to the case of the
Heisenberg group \cite{MaSt} (see also \cite{Dr}) and to the present metric setting
setting by \cite{LoVi2} for quadratic Hamiltonians. We will show, that under further
assumptions on the space this holds also in the metric setting for general Hamiltonians.
Namely, we prove the following theorem in Section \ref{proofs}. Notice, that here $\mu$ needs not to be a probability measure.

\begin{thm}\label{solutionHJ}
\begin{enumerate} \item[(i)] The infimum in \eqref{HLX} is attained.
\item[(ii)]For $0\leq s < t$ we have the semigroup property
\begin{equation*}
Q_tg(x)=\min_{y\in X}\Big[(t-s)L
\bigg(\frac{d(x,y)}{t-s}\bigg)+Q_sg(y)\Big]
\end{equation*}
for all $x\in X.$ \item[(iii)] For all $x\in X$, $Q_tg(x)$ is non--increasing in $t$.
\item[(iv)] $(x,t)\mapsto Q_{t}g(x)$ is in  $\mathrm{ Lip}(X\times \R_+)$.
\item[(v)] For all $x\in X$, $u(x,t)=Q_tg(x)$ solves \eqref{HJEX}
for a.e. $t>0$. \item[(vi)] For every $x\in X$ and
$t>0$\begin{equation}\label{liminf} \liminf_{s\to
0^{+}}\frac{Q_{t+s}g(x)-Q_tg(x)}{s}\geq -H\big(|\nabla^-
Q_tg|(x)\big).
\end{equation}
\item[(vii)] If $(X,d,\mu)$ supports a local Poincar\'e inequality
and $\mu$ is doubling, then
\begin{equation*}
\limsup_{s\to 0^+} \frac { Q_{t+s}g(x) - Q_tg(x) }s \le  - H\big( |
\nabla^- Q_tg|(x)\big)
\end{equation*}
for all $t>0$ and $\mu$--a.e. $x\in X$. \item[(viii)] If
$(X,d,\mu)$ supports a local Poincar\'e inequality and
$\mu$ is doubling, $u(x,t)=Q_tg(x)$ solves \eqref{HJEX} for all $t>0$
and for $\mu$--a.e. $x\in X$.
\end{enumerate}
\end{thm}

%Our first statement concerns the Lipschitz regularity of $u$:
%Our first observation is that the function given by \eqref{HLX} is
%Lipschitz continuous.
%\begin{thm}\label{regularity}
%Let $(X,d)$ be proper and geodesic. Assume that  $g\colon
%X\rightarrow\R$ is Lipschitz continuous, and the Hamiltonian
%$H\colon\R_+\rightarrow\R_+$ satisfies $H(0)=0$, is increasing,
%convex and superlinear. Then the function $u$ given by (\ref{HLX})
%is Lipschitz continuous on $X\times\R_+$. In particular, we  have
%$lip(u)\leq \max\{lip(g), H(lip(g))\}$.
%\end{thm}
%Here we used the notation
%\[
%lip(g)= \sup_{x\neq y}\frac{| g(x)-g(y)|}{d(x,y)},
%\]
%for the Lipschitz constant of $g\colon X\to \R$. We define the
%Lipschitz constant of \mbox{$u\colon X\times \R_+ \to \R$} in the
%same way using the canonical product  metric
%\[
%d_{X\times\R_+}((x,t), (y,s))=  d(x,y) + | s-t|.
%\]
%We recall that by Rademacher's theorem the Lipschitz continuity of
%$u$ implies differentiability of $u$ a.e. in the $t$ variable.

\section{Logarithmic Sobolev inequalities and hypercontractivity of the Hamilton--Jacobi semigroup.}
\label{s-hc}
\subsection{Logarithmic Sobolev inequality}

The $q$--logarithmic--Sobolev inequality is a quantitative
expression of  the fact that the entropy of a function is
dominated by the $q$--norm of its gradient.  The entropy
functional for an integrable, non-negative function $h:X\to \R_+$
is defined by
\begin{equation} \label{ent}
\ent{\mu}{h} = \int_X h\log h \, d\mu - \int_X h\, d\mu\,\log\int_X
h\,d\mu .
\end{equation}

\begin{defn} \label{logsob}
If $K>0$ and $1< q \leq2$ we say that $(X,d,\mu)$ satisfies a
\textit{q--log--Sobolev inequality} with a constant $K$, $q{\text
-LSI(K)}$, if for any Lipschitz function $f$ we have
\begin{equation}\label{log-sobolev}
\ent{\mu}{|f|^q} \le (q-1) \left(\frac qK\right)^{q-1} \int_X {|\nabla^- f|}
^q\,d\mu.
\end{equation}
\end{defn}

Notice, that for $q>2$ it is not possible to have
\eqref{log-sobolev}, as for $f=1+\varepsilon g$, where
$\varepsilon \to 0$, the left--hand side behaves like
$\varepsilon^2$ where as the right--hand side like
$\varepsilon^q$; see \cite{BaKo}. Notice also, that Corollary 3.2.
in \cite{BaKo} provides an example of a measure that satisfies
\eqref{log-sobolev}.

\subsection{Hypercontractivity of the Hamilton-Jacobi semigroup}

The equivalence between the hypercontractivity of the quadratic Hamilton-Jacobi 
semigroup and the logarithmic Sobolev inequality in $\R^n$
%A similar problem for the quadratic semigroup in $\R^n$  
is established in
\cite{BobGenLed}, and our approach follows the same lines.

Let $\mu$ be a probability measure on the Borel sets of $\R^n$. We
will denote by $\|\cdot \|_p$, $p\ge 1$, the $L^p$-norm with
respect to $\mu$. Bobkov et al.~\cite{BobGenLed} have shown that a
measure $\mu$ which is absolutely continuous with respect to the
Lebesgue measure satisfies the classical logarithmic Sobolev
inequality with constant $\rho$ if and only if the Hamilton-Jacobi
semigroup $Q_t$ associated to the quadratic inf-convolution
operator is hypercontractive, i.e. we have
\begin{equation}\label{hyperc}
\| e^{Q_t f} \|_{a+\rho t} \le \|e^f\|_a
\end{equation}
for every bounded measurable function $f$ on $\R^n$, every $t\ge 0
$ and every $a\in \R$. The strategy of the proof, going back to
Gross, consists of studying the monotonicity properties of the
left hand side of  ~(\ref{hyperc}) by differentiating with respect
to $t$.

\subsection{Hypercontractivity and Log--Sobolev inequality}
In this section we prove the equivalence between the
$q$--logarithmic Sobolev inequality and the hypercontractivity of
the corresponding Hamilton--Jacobi semigroup.
To state our result we impose additional conditions on the space
$X$ which guarantee that the infimum--convolution $Q_t f$ solves
the Hamilton-Jacobi equation for a Lipschitz initial--value
function $f$. We consider the Hamilton--Jacobi equation on $X$
with the Hamiltonian $H(v) = v^q/q$, which corresponds to $L(u) =
u^p/p$.

\begin{thm}\label{hypercontract}
Suppose that $(X,d,\mu)$ supports a local $(1,s)$--Poincar\'e
inequality for some $s\ge 1$, %$X$ is {\bf [compact ?]} 
and $\mu$ is
doubling. Furthermore, assume that $(X,d,\mu)$ satisfies the
$q$--logarithmic Sobolev inequality with some constant $K$, and that 
$a$, $\rho>0$ are related by the inequality
\begin{equation}\label{consts}
a^{2-q}K^{q-1} \ge \rho (q-1).
\end{equation}
Then for every bounded measurable function $f$ on $X$ and
every $t\ge 0$
\begin{equation}\label{hc}
\| e^{Q_t f}\|_{a+ \rho t}\le \|e^{f}\|_a.
\end{equation}
Conversely, if~(\ref{hc}) holds for all $t\ge 0$, then the
$q$--logarithmic Sobolev inequality, $q{\text -LSI(K_0)}$, holds on $X$
with a constant $K_0$ which satisfies~(\ref{consts}) with an equality.
\end{thm}

\begin{proof}
Let $F(t) = \|e^{Q_t f}\|_{\lambda(t)}$ with $\lambda (t) = a +
\rho t$, $t>0$. For all $t>0$, $\frac{\partial}{\partial t} Q_t
f(x)$ exists. Hence, $F(t)$ is differentiable at every point $t>0$, and we
get
\begin{equation}\label{Fineq}
\lambda^2(t) F(t)^{\lambda(t) -1} F'(t) = \rho \ent{\mu}{e^{\lambda(t)\, Q_t f}}+
\int_X \lambda^2(t) \frac\partial{\partial t} Q_t f \,e^{\lambda(t)\, Q_t f}\,d\mu.
\end{equation}
Since  $\frac{\partial}{\partial t} Q_t
f(x)=-|\nabla^- Q_tf(x)|^q/q$ $\mu$--a.e. in
$X$ by Theorem~\ref{solutionHJ} (vii), we have
\begin{equation*}
\lambda^2(t) F(t)^{\lambda(t) -1} F'(t) =\rho \ent{\mu}{e^{\lambda(t)\, Q_t f}} -  \lambda^2(t) \int_X \frac {|\nabla^-Q_t f|^q}q  e^{\lambda(t)\,
Q_tf } \, d\mu.
\end{equation*}

Since $Q_tf (x)$ is Lipschitz continuous, we can apply the
$q$--logarithmic Sobolev inequality to $e^{\lambda(t)\, Q_t f}$ to
deduce that $F'(t)\le 0$ for all $t>0$. Since $F(t)$ is continuous
it is non--increasing.

To prove the converse, consider a Lipschitz continuous function
$f$. Then (\ref{hc}) implies $F'(0)\le 0$. The Hamilton-Jacobi equation implies 
$$\left.\frac{\partial}{\partial t} Q_t
f(x)\right|_{t=0}=-|\nabla^- f(x)|^q/q$$ $\mu$--a.e. in
$X$. Thus regarding~(\ref{Fineq}) at $t=0$, we get
% \marginpar{we need to state / prove this somewhere}
%{\bf As $t\to 0$, $Q_t f $ converges monotonically to $f$}, so that when passing to the limit $t\to 0$ in~(\ref{Fineq}) we get
\begin{equation}
\rho \ent{\mu}{e^{af}} \le a^2 \int_X e^{af} \frac {|\nabla^-f|^q}q\, d\mu.
\end{equation}
By setting $e^{af} = g^q$ this leads to the $K_0$--logarithmic
Sobolev inequality, where $K_0$ satisfies~(\ref{consts}) with an equality.
\end{proof}

\begin{rem}
The hypercontractivity of the infimum convolution semigroup
holds only  for $q\le 2$.
\end{rem}
\begin{proof}
Indeed, suppose that $q>2$ and consider a bounded non--negative
function $f$ with $\mathop{\mathrm{essup}}_X f > \int_X f\,d\mu$.
Fix a small $\delta>0$.

Since $q>2$, it is possible to
choose $t\to \infty$, $\epsilon\to 0$ so that $\epsilon^{q-1} t = \delta$ and $\epsilon t \to \infty$.
Directly from the definition one can check that the scaling property of $Q_t$, namely
$$Q_t(\epsilon \, f)(x) = \epsilon \, \left( Q_{\epsilon^{q-1}t} f\right)(x)$$
holds for all $x\in X$ and $t,\epsilon >0$.

Then we get from~(\ref{hc}) that
$$ \| e^{Q_t (\epsilon f)}\|_{a+ \rho t}^{ 1/\epsilon}= \| e^{Q_\delta f}\|_{(a+ \rho t)\epsilon}  \le \|e^{\epsilon f}\|_a^{ 1 /\epsilon}=
\|e^f\|_{a\epsilon},$$
whence
$$e^{\mathop{\mathrm{essup}}_X (Q_\delta f)}\le e^{\int_X f\,d\mu}.$$
Letting $\delta\to 0$ we obtain a contradiction.
\end{proof}

\section{Talagrand and logarithmic Sobolev inequalities}\label{s-t}
\subsection{Wasserstein distance and the Talagrand inequality}

Let $1\leq p<\infty$. The $p$--Wasserstein distance between two
probability measures on $X$ is defined as
\begin{equation}\label{w}
  W_p(\mu,\nu)=\left(\inf\iint\frac1p{d(x,y)}^p\:d\pi(x,y)\right)^{1/p},
   \end{equation}
where the infimum is taken over all probability measures $\pi$ on
$X\times X$ with marginals $\mu$ and $\nu$. By the
Monge--Kantorovitch dual characterization, see~\cite{Ra}, we can
write
\begin{equation}\label{mk}
  {W_p(\mu,\nu)}^p = \sup \left[\int_X g\, d\nu - \int_X f\, d\mu\right],
\end{equation}
where the supremum is taken over all pairs $(f,g)$ of bounded
measurable functions such that for all $x$ and $y$ we have
\begin{equation}\label{leq}
g(x) \le  f(y) + \frac{{d(x,y)}^p}p .
\end{equation}

%\subsection{Talagrand inequality}
Recall that the entropy functional for an integrable,
non--negative function was defined in \eqref{ent} in the previous
section.
\begin{defn}
Let $p\geq 2$. We say that $(X,d, \mu )$ satisfies the {\it
$p$--Talagrand inequality} with a constant $K$, $p{\text
-Tal}(K)$, if for any probability measure $\nu\ll \mu$ on $X$
there holds
\begin{equation}\label{t}
  {W_p(\nu, \mu)}^p \le \frac 1K\ent{\mu}{\frac {d\nu}{d\mu}}.
\end{equation}
\end{defn}
Let us mention that our definition differs from  the standard version of the Talagrand
inequality defined for $1\le p\le 2$, namely
\begin{equation*}
  {W_p(\nu, \mu)}^2 \le \frac 1K\ent{\mu}{\frac {d\nu}{d\mu}},
\end{equation*}
which has been widely studied in the literature, see e.g.~\cite[Chapter 22]{Vil}.
As we shall show in Theorem \ref{SobTalPoi} below, the version~\eqref{t}
is equivalent to the appropriate $q$--logarithmic Sobolev inequality.

Notice, that if $d\nu /d\mu$ is of the form $1+ \varepsilon \:g$
where $\varepsilon \to 0$, then $\ent{\mu}{d\nu/ d\mu}$ is of
order $\varepsilon^2$, whereas ${ W_p(\nu,\mu)}^p$ is typically of
order $\varepsilon^p$ as the following example shows.

Let $(M, vol)$ be a smooth compact connected Riemannian manifold and let $\mu$ and $\nu$ be
two probability measures absolutely continuous with respect to $vol$, considered as elements
of the Wasserstein space of probability measures on $M$ with quadratic distance $W_2$. It is known (see~\cite{McCann}) that
there is a unique geodesic $\mu_\epsilon$ (with respect to $W_2$) in the Wasserstein space that joins
$\mu$ and $\nu$. Moreover, the measure is transported along the geodesics in $M$ in the following way.
There exists a family of maps $\{F_\epsilon\}_{\epsilon\in [0,1]} \colon M\to M$ such that $\mu_\epsilon = (F_\epsilon)_*\mu_0$.
%and the measure $(\mathrm{Id}, F_t)_* \mu_0$ on $M\times M$ realizes the optimal coupling of $\mu_0$ and $\mu_t$.
More precisely, for almost all $m\in M$, $F_\epsilon(m) = \exp_m(-
\epsilon\nabla\phi(m))$ for a certain Lipschitz continuous
function $\phi$ on $M$ with an almost everywhere defined Hessian (see~\cite{RiemInterp}).
It follows that for small $\epsilon$ we have
$$\mu_\epsilon(dm) = \mu_0(dm) \, \big(1+ \epsilon\Delta \phi (m) +o(\epsilon) \big) .$$
Consider the coupling
$(\mathrm{Id}, F_\epsilon)_* \mu_0$ of $\mu_0$ and $\mu_\epsilon$. Then
\begin{equation} \begin{split}
W_p(\mu_0,\mu_\epsilon)^p & = \inf_{\pi }\int_{M\times M} \frac {d(x,y)^p}p \, d\pi(x,y) \\
& \le
\int_{M\times M} \frac {d(x,y)^p}p \, d\big((\mathrm{Id}, F_\epsilon)_* \mu_0\big)(x,y)  \\&
= \int_M \frac{
d(m, F_\epsilon(m))^p}{p} \,d\mu_0(m) \\ & = \epsilon^p \int_M \frac{|\nabla \phi(m)|^p}{p}\, d\mu_0(m).
  \end{split}
\end{equation}
Thus \eqref{t} does not hold for $1\leq p
< 2$.
\subsection{The dual formulation of the Talagrand inequality}

To establish a connection between the Talagrand and the
log--Sobolev inequality, we have to consider the dual formulation
of the Talagrand inequality using the Hamilton--Jacobi semigroup.
For an arbitrary function $f$ on $X$, consider the infimum
convolution \eqref{HLX} with Lagrangian $L(u)=u^p/p$, namely
%we define
%the infimum convolution by
%\begin{equation}\label{infconv}
$$Q_t f(x) = \inf_{y\in X} \left[  \frac { {d(x,y)}^p}{p\,t^{p-1}}+f(y)\right],$$
and write $Qf$ for $Q_1f$.
%Note that $Q_t f$ equals $u(x,t)$
%from~(\ref{HLX}) for a particular choice of $L$. Throughout the
%paper we will use both notations giving preference to $u(x,t)$ to
%denote the solutions of the Hamilton--Jacobi equation.
Following~\cite[Section 3.3]{BobGenLed}, we notice that by the
Monge--Kantorovitch duality \eqref{mk} and~(\ref{leq}), the $p$--Talagrand
inequality is equivalent to
\begin{equation}\label{infconvTal}
  \int_X \left(Qf - \textstyle{\int_X f\, d\mu}   \right){\frac {d\nu}{d\mu}}\, d\mu \le \frac 1K\ent{\mu}{\frac {d\nu}{d\mu}},
\end{equation}
for every bounded function $f$. Define two functions: $\psi_0 :
=K\left(Qf - \textstyle{\int_X f\, d\mu}\right) $ and $\phi :=
\frac {d\nu}{d\mu}$. Recall that by the variational
characterization of the entropy $$\ent{\mu}{\phi} = \sup_{\int_X
e^\psi \,d\mu \le 1} \int_X \psi\phi\, d\mu.$$ Indeed, the
left--hand side is smaller than or equal to the right--hand side
by definition. The converse inequality results from Jensen's
inequality applied to the convex function $ x\mapsto x\log x$ and the
probability measure  ${e^\psi}d\mu/{\int_X e^\psi d\mu}$.

Since~(\ref{infconvTal}) holds for every choice of
$\frac{d\nu}{d\mu}$,
%by a similar argument as in ~\cite[Section 3.3]{BobGenLed},
it is therefore equivalent to $\int_X e^{\psi_0}\,d\mu \le 1$, i.e.
\begin{equation}\label{dualT}
  \int_X e^{K \: Qf}\, d\mu \le e^{K\int_X f\, d\mu} .
\end{equation}
The latter inequality is known as the dual form of the
$p$--Talagrand inequality.
%For more details regarding the
%derivation of the dual form of the Talagrand inequality we refer
%the reader to~\cite[Section 3.3]{BobGenLed}.

%
%We say that $(X,d,\mu)$ satisfies a \textit{global q--Poincar\'e
%inequality} with constant $K$ if for any Lipschitz function $f$
%with $\int_Xf\,d\mu=0$ we have
%\begin{equation}\label{globalpoincare}
% \int_{X}|f|^q \,d\mu\leq  \frac{1}{K} \int_X {|\nabla^- f|}^q\,d\mu.
%\end{equation}
%{\bf
%\begin{itemize}
%\item Definition of a local Poincare inequality or just a
%reference?
%\end{itemize}
%}
%Finally, we say that $(X,d,\mu)$ satisfies a \textit{local
%q--Poincar\'e inequality} with constant $C>0$ if for any Lipschitz
%function $f$ with $\int_Xf\,d\mu=0$ we have
%\begin{equation}\label{globalpoincare}
% \int_{X}|f|^q \,d\mu\leq  \frac{1}{K} \int_X {|\nabla^- f|}^q\,d\mu.
%\end{equation}

%For $2\geq q>1$, let $p>2$ be its conjugate, so that $1/p+1/q=1$.
%Fore this section we introduce the notation
%\begin{equation*}
%Q_tg(x)=\inf_{y\in\R^n}\Big\{tL\bigg(\frac{x-y}{t}\bigg)+g(y)\Big\}
%\end{equation*}
%for the Hopf--Lax transform.
\subsection{Talagrand and log-Sobolev inequality}

In order to state the main result of this section we need to
recall one more concept, the notion of displacement convexity from
\cite{LoVi}. Recall that for $p\in [1,\infty)$ the space $P_p(X)$
of Borel probability measures on a compact length space $X$ with the Wasserstein distance
$W_p$ is itself a compact length space, see~\cite[Remark 2.8]{LoVi}. If $\nu$
is a probability measure which is absolutely continuous with
respect to $\mu$, we define the entropy functional $U_\mu$ on
$P_p(X)$ by
\begin{equation}\label{entW}
U_\mu(\nu) = \int_X \frac {d\nu}{d\mu} \log\left(\frac {d\nu}{d\mu} \right)\,d\mu = \ent{\mu}{\frac {d\nu}{d\mu}}.
\end{equation}
Following~\cite{LoVi}, we say that it is {\it weakly displacement
convex} if for all $\nu_0$, $\nu_1\in P_p(X)$, there is some
Wasserstein geodesic $\{\nu_t\}_{t\in [0,1]}$ from $\nu_0$ to
$\nu_1$ along which
\begin{equation}\label{conv}
U_\mu(\nu_t) \le  t U_\mu(\nu_1) + (1-t) U_\mu (\nu_0).
\end{equation}
Notice also, that in the Riemannian setting, Villani considers a version of Theorem~\ref{SobTalPoi} with a different choice of Lagrangian; see \cite{Vil}[Thm 22.28].

\begin{thm}\label{SobTalPoi} Let $2\geq q>1$ and $p\ge 2$ be its conjugate, so that $1/p+1/q=1$.
%Assume that $(X,d,\mu)$ is a  %{\bf [compact ?]} 
%geodesic metric measure space.
\begin{enumerate}
\item[(i)] Let $(X,d,\mu)$ satisfy the $p$--Talagrand inequality
with some constant $K>0$, and assume that $X$ is compact. If the entropy functional
$U_\mu(\cdot)$ is weakly displacement
convex  then $(X,d,\mu)$ also satisfies the $q$-logarithmic
Sobolev inequality with the constant $K p^{-p}$.

\item[(ii)] Suppose that $(X,d,\mu)$ supports a local
$(1,s)$--Poincar\'e inequality for some $s\ge 1$, and $\mu$ is
doubling. Then, if $(X,d,\mu)$ satisfies the $q$--logarithmic
Sobolev inequality with some constant $K>0$, then it also
satisfies the $p$--Talagrand inequality with the same constant.
\end{enumerate}
\end{thm}

\begin{proof}
%Let us start by proving $(i)$. Let $h\in \mathrm{Lip} (X) $ satisfy $\int_X h \, d\mu =0$.
%Introduce
%\begin{equation}
%\psi( t) = \int_X e^{K \, t^{p-1} Q_t h}\, d\mu.
%\end{equation}
%From Talagrand's inequality in its dual formulation~(\ref{dualT}),
%we deduce that $\psi(t) \le \exp(K t^{p-1} \int_X h\, d\mu) =1$.
%Hence, $\psi(t)$ has a maximum at $t=0$. Combining this with $\int
%h\, d\mu = 0$, we find:
%\begin{equation}\label{}
%    0 \le \limsup_{t\to 0^+} \left( \frac {1-\psi(t)}{ K t^p}\right)
%    =\limsup_{t\to 0^+} \int_X \left(\frac{ 1+ K t^{p-1} \, h - e^{K t^{p-1} Q_th}}{Kt^p}\right) \, d\mu.
%\end{equation}
%We know that $Q_t h$ is a non-increasing function of $t$ that
%converges monotonically to $h$ as $t\to 0^+$. By the boundedness of
%$Q_t h$,
%\begin{equation}\label{}
%    e^{K t^{p-1} Q_t h }= 1 + Kt^{p-1} Q_t h  +\frac{ K^2 t^{2(p-1)} } 2 h^2  + o( t^{2(p-1)}).
%\end{equation}

%Since $(h-Q_t h)/t$ is bounded, we are allowed to apply the
%dominated convergence theorem in the form
%\begin{equation}
% \limsup_{t\to 0^+}\int_X \left( \frac {h-Q_t h} t\right)\, d\mu \le \int_X \limsup_{t\to 0^+}\left( \frac{h-Q_t h} t\right)\, d\mu.
%\end{equation}
%In view of Proposition~\ref{prop:liminf}, this implies that
%\begin{equation}
% \limsup_{t\to 0^+}\int_X \left( \frac {h-Q_t h} t\right)\, d\mu \le \int_X \frac {|\nabla^- h|^q}q.
%Z\end{equation}

Consider a probability measure $\nu$ on $X$ with a positive
Lipschitz continuous density function $f$ with respect to $\mu$.
Then from~\cite[Proposition 3.36]{LoVi} it can be easily deduced
that
\begin{equation}\label{coupl}
U_\mu(\nu) \le \int_{X\times X} \frac {|\nabla^- f(x_0)|}{f(x_0)} \, d(x_0,x_1)\, d\pi(x_0\, x_1),
\end{equation}
where $\pi$ is the optimal coupling of $(\nu, \mu)$.
Applying the H\"older inequality on the right-hand side gives
\begin{equation}
U_\mu(\nu)  \le    p^{1/p}\, W_p(\mu, \nu)\left( \int_X \frac {|\nabla^- f(x_0)|^q}{f^{q-1}(x_0)} \, d\mu(x_0)
\right)^{1/q} .
\end{equation}
Hence the $p$--Talagrand inequality implies
\begin{equation}
U_\mu(\nu) = \ent{\mu}{f}\le  \left(\frac pK\right)^{q/p}\int_X \frac {|\nabla^- f|^q}{f^{q-1}}\, d\mu.
\end{equation}
Replacing $f$ with $|g|^q$ we arrive at the $q$--logarithmic
Sobolev inequality, $q{\text -LSI(K p^{-p})}$, with the desired
constant. This proves $(i)$.

\bigskip
%\marginpar{{\bf Here we need to check the uniformity claims!}} 
To prove $(ii)$ we follow the idea in \cite{LoVi2}. We consider the
Hamilton--Jacobi equation on $X$ with the Hamiltonian $H(v) =
v^q/q$, which corresponds to $L(u) = u^p/p$ and the associated
semigroup \eqref{HLX} $Qf= Q_1f$. From the Talagrand inequality in
its dual formulation \eqref{dualT} it follows that it is
sufficient to show that
\begin{equation}
\int_X e^{K \: Qf}\, d\mu \le e^{K\int_X f\, d\mu}
\end{equation}
for every continuous bounded function $f$. Set, for some $n \ge 1$,
\begin{equation}
\phi(t) = \frac 1 {Kt^n} \log\left( \int_X e^{K t^n Q_t f}\,
d\mu\right).
\end{equation}
Since $f$ is bounded, we know that $Q_t f $ is bounded
uniformly in $t$. Thus
\begin{equation}
\int_X e^{K t^n Q_t f}\, d\mu =1 + Kt^n \int_X Q_t f \, d\mu +
O(t^{2n}),
\end{equation}
and
\begin{equation}
\phi(t) = \int_X Q_t f\, d\mu +O(t^n).
\end{equation}
%By the {\bf uniform convergence of $Q_t f$ to $f$ as $t\to 0^+$}
Since $Q_t f\to f$ as $t\to 0^+$,
we have by the dominated convergence theorem that
\begin{equation}
\lim_{t\to 0^+} \phi(t) = \int_X f\, d\mu.
\end{equation}
Therefore, our goal is to prove that $\phi(1)\le \lim_{t\to 0^+}
\phi(t)$. For this, it suffices to prove that $\phi(t) $ is
non--increasing in $t$. Let us fix $t\in(0,1]$. For $s>0$, we have
\begin{equation}
\begin{split}
\frac {\phi(t+s) - \phi(t) } s =& \frac 1s \left( \frac
1{K(t+s)^n} - \frac 1{K t^n}\right)
\log\int_X e^{K(t+s)^n Q_{t+s} f}\, d\mu \\
&+ \frac 1 {Kt^n s} \left( \log\int_X e^{K(t+s)^n Q_{t+s} f}\,d\mu
- \log \int_X e^{K t^nQ_t f}\, d\mu\right).
\end{split}
\end{equation}
As $s\to 0^+$, the first term on the right--hand side converges to
\begin{equation}
-\frac n{K \, t^{n+1}}\log \left( \int_X e^{K t^n Q_t f }\, d\mu\right).
\end{equation}
The limit of the second term, provided it exists, is
\begin{equation}
\frac 1{K t^n} \frac 1{\int_X e^{K t^n Q_t f}\, d\mu}\:
\lim_{s\to 0^+}  \left[\frac 1s \left(\int_X e^{K (t+s)^n Q_{t+s}
f}\, d\mu  - \int_X e^{K t^n Q_t f}\,d\mu\right) \right].
\end{equation}
The expression in brackets can be written as
\begin{equation}\label{limarr}
\int_X \left( \frac {e^{K(t+s)^n Q_{t+s} f} - e^{Kt^n Q_{t+s} f}} s\right)\, d\mu
+\int_X \left( \frac { e^{Kt^n Q_{t+s} f} - e^{Kt^n Q_t f} }
s\right)\, d\mu.
\end{equation}
The first term in~(\ref{limarr}) has the form $e^{K t^n Q_{t+s}
f}(e^{K(nt^{n-1} s+o(s))Q_{t+s} f} -1)/s$ so it converges
%uniformly
to $(e^{K t^{n}Q_t f}) Knt^{n-1}Q_t f$ as $s\to 0^+$.
By the dominated convergence theorem the first integral in~(\ref{limarr}) thus converges to
$$\int_X Knt^{n-1} Q_t f e^{K t^n Q_t f}\, d\mu.$$

Let us now consider the second term of~(\ref{limarr}). By Theorem
\ref{solutionHJ}  (vi) and (vii), for $\mu$--a.e. $x\in X$ we have
\begin{equation}
Q_{t+s}f(x) = Q_t f(x) - s\left(\frac {{|\nabla^-Q_t f(x)|}^q}q +o(1)\right),
\end{equation}
and therefore
\begin{equation}\label{limDC}
\lim_{s\to 0^+}\frac{ e^{K t^n Q_{t+s} f} - e^{K t^n Q_t f}} s =
-Kt^n e^{Kt^n Q_t f} \frac {{|\nabla^-Q_t f|}^q}q.
\end{equation}
On the other hand, as $Q_{(\cdot)}g(\cdot)$ is Lipschitz on
$X\times \R_+$, $Q_{t+s} f= Q_t f + O(s)$ holds uniformly on
$X$. Since $Q_t f(x)$ is uniformly bounded in $x$, we deduce
that
\begin{equation}\label{forDC}
\frac{e^{K t^n Q_{t+s}f} - e^{K t^n Q_t f}} s =O(1)
\end{equation}
as $s\to 0^+$. In view of~(\ref{limDC}) and ~(\ref{forDC}) we
apply the dominated convergence theorem to compute the
limit of the second integral in~(\ref{limarr}), that is,
\begin{equation}
\lim_{s\to 0^+}\int_X\left(\frac {e^{Kt^n Q_{t+s} f} -
e^{Kt^n Q_t f}} s\right)\,d\mu = - K t^n\int_X \frac{{|\nabla^-
Q_t f|^q}}q e^{Kt^n Q_t f}\, d\mu.
\end{equation}
In summary, we have
\begin{equation}\label{final}
\begin{split}
\lim_{s\to 0^+}&\left[ \frac{\phi(t+s)-\phi(t)}s\right]=\\
&\frac 1{Kt^{n+1}\int_X e^{Kt^n Q_t f}\,d\mu} \bigg[-n\log
\left(\int_X e^{Kt^n Q_t f}\, d\mu\right) \int_X e^{K t^n Q_t
f}\,d\mu  \\ & + \int_X n  Kt^n Q_t f \:e^{Kt^n Q_t f} \,d\mu
-\int_X Kt^{n+1} \frac{{|\nabla^- Q_t f|}^q}q\: e^{K t^n Q_t
f}\,d\mu\bigg].
\end{split}
\end{equation}

Recall that for $q\in (1,2]$, the $q$--logarithmic Sobolev
inequality with constant $K$ states that for every Lipschitz
function $g$ on $X$
\begin{equation}\label{sobolev}
\ent{\mu}{|g|^q} \le (q-1) \left(\frac qK\right)^{q-1} \int_X {|\nabla^- g|}
^q\, d\mu.
\end{equation}
Set $n= 1/(q-1)$. Applying~(\ref{sobolev}) with $g = \exp
\left(Kt^n Q_tf/q \right)$ shows that~(\ref{final}) is
non--positive, and $(ii)$ follows.
\end{proof}

\begin{rem}
Let p=q=2. In the setting of Riemannian manifolds, i.e. when $X=(M, vol)$,
%\marginpar{A reference to these notions} assumption of
the displacement convexity in the first part of
Theorem~\ref{SobTalPoi} is verified if the reference measure $\mu
=e^{-V}\, vol$, with $\mu(M)=1$ and $V\in C^2(M)$, satisfies the curvature-dimension
$CD(0,\infty)$ inequality; see~\cite{LoVi}.
\end{rem}

\section{Solutions to Hamilton--Jacobi equation}
\label{proofs}

%Recall that $(X,d)$ is proper, $d$ is a length metric, $g\colon
%X\rightarrow\R$ is Lipschitz continuous, and the Hamiltonian
%$H\colon\R_+\rightarrow\R_+$ satisfies $H(0)=0$, is increasing,
%convex and superlinear.

%Our first result is that the infimum in \eqref{HLX} is attained
%(see Lemma 2.4 in \cite{Dr}).
%\begin{lem} \label{lem: infminX}
%For a given $t>0$, the infimum in (\ref{HLX}) is attained by some
%$y\in X$ and therefore
%\begin{equation*}
%Q_tg(x)=\min_{y\in
%X}\Big\{tL\bigg(\frac{d(x,y)}{t}\bigg)+g(y)\Big\}.
%\end{equation*}
%\end{lem}

\begin{proof}[Proof of Theorem \ref{solutionHJ} (i)] Fix $x\in X$ and $t > 0$. Notice, that by choosing $y=x$ in \eqref{HLX} we
get $Q_tg(x) \leq g(x)$.
%\begin{equation*}
%Q_tg(x)=\inf_{y\in
%X}\Big\{tL\bigg(\frac{d(x,y)}{t}\bigg)+g(y)\Big\} \leq g(x).
%\end{equation*}

Let $(y_n)$ be a minimizing sequence in \eqref{HLX} and assume
first that it is bounded.
%Assume first that $(y_n)\subset X$ is
%bounded, i.e. $(y_n)\subset \overline{B(x,R)}$ for some $R>0$.
Since $X$ is proper there exists $y_0\in X$ and a subsequence
$(y_{n_k})$ such that $y_{n_k} \to y_0$, whence the continuity of
$L$ and $g$ imply that
$$Q_tg(x)=\lim_{k\to \infty}\Big\{tL\bigg(\frac{d(x,y_{n_k})}{t}\bigg) +g(y_{n_k})\Big\}= tL\bigg(\frac{d(x,y_0)}{t}\bigg) +g(y_0).$$
On the other hand, if $\lim_{n\to \infty} d(y_n,x) \to
\infty$, the %We shall show below that this leads to a contradiction.
superlinearity of $L$ implies for any $M>0$ we have
$$L\bigg(\frac{d(x,y_n)}{t}\bigg)\geq M \frac{d(x,y_n)}{t}$$
for $n$ large enough. Multiplying the above  inequality by $t$ and
adding $g(y_n)$ on both sides we get
$$tL\bigg(\frac{d(x,y_n)}{t}\bigg)+ g(y_n) \geq M d(x,y_n) +g(y_n) \geq (M-lip(g))d(x,y_n)
-|g(x)|,$$ since $g$ is Lipschitz. Choosing $M:=lip(g) +1$ we
obtain
$$tL\bigg(\frac{d(x,y_n)}{t}\bigg) + g(y_n) \geq d(x,y_n) -|g(x)|,$$
which implies that
$$\lim_{n\to \infty}tL\bigg(\frac{d(x,y_n)}{t}\bigg) + g(y_n) = \infty, $$
which is a contradiction. Hence $(y_n)$ is bounded and the infimum
in \eqref{HLX} is attained.
\end{proof}

%Next result is the so called semi-group property. % of the inf
%convolution \eqref{HLX} giving the expression of $Q_tg(x)$.

%SEMI-GROUP PROPERTY=================================================

%\begin{prop}%[Semi-group property]
%\label{prop: semigroup}
%Let $0\leq s < t$ and $u$ be given by (\ref{HLX}). We have
%\begin{equation}
%u(p,t)=\min_{q\in X}\Big\{(t-s)L
%\bigg(\frac{d(p,q)}{t-s}\bigg)+Q_sg(q)\Big\} \ \mbox{for all} \
%p\in X. \label{semigrouppropgeospace}
%\end{equation}
%\end{prop}

\begin{proof}[Proof of Theorem \ref{solutionHJ} (ii)]%We start by showing that
%\begin{equation}
%Q_tg(p) \leq (t-s)L\bigg(\frac{d(q,p)}{t-s}\bigg)+Q_sg(q)
%\label{1.Ungl.d.SGP}
%\end{equation}
%for all $q\in X$. To this end let us
Fix $q\in X$. By (i) there exists a $v\in X$ such that
\begin{equation*}
Q_sg(q)=sL\bigg(\frac{d(v,q)}{s}\bigg)+g(v).
\end{equation*}
%Using the triangle inequality and monotonicity of $L$ we get
%\begin{equation*} L\bigg(\frac{d(v,p)}{t}\bigg)\leq
%L\bigg(\frac{d(v,q)+d(q,p)}{t}\bigg).
%\end{equation*}
Set $\tau:=\frac{s}{t}$, $\sigma:=\frac{t-s}{t}$, and use the
monotonicity and convexity of $L$ to obtain
\begin{equation*}
\begin{split}
L\bigg(\frac{d(v,p)}{t}\bigg) &\leq L\bigg(\tau\frac{d(v,q)}{\tau
t}+\sigma\frac{d(q,p)}{\sigma t}\bigg)
                          \leq  \tau L\bigg(\frac{d(v,q)}{\tau t}\bigg)+\sigma L\bigg(\frac{d(q,p)}{\sigma t}\bigg)\\
                             &=  \frac{s}{t}L\bigg(\frac{d(v,q)}{s}\bigg)+ \frac{t-s}{t}L\bigg(\frac{d(q,p)}{t-s}\bigg).
\end{split}
\end{equation*}
Multiplying the inequality by $t$ and adding $g(v)$ on both sides
yields
\begin{equation*}
\begin{split}
Q_tg(p) \leq tL\bigg(\frac{d(v,p)}{t}\bigg)+g(v) & \leq  (t-s)L\bigg(\frac{d(q,p)}{t-s}\bigg) + sL\bigg(\frac{d(v,q)}{s}\bigg)+g(v) \\%\underbrace{sL(\frac{d(v,q)}{s})+g(v)}_{=Q_sg(q)}\\
                             & =  (t-s)L\bigg(\frac{d(q,p)}{t-s}\bigg) + Q_sg(q).
\end{split}
\end{equation*}
Since $q\in X$ is  arbitrary we obtain
\begin{equation*}
Q_tg(p)\leq \min_{q\in
X}\Big\{(t-s)L\bigg(\frac{d(q,p)}{t-s}\bigg)+Q_sg(q)\Big\}.
\end{equation*}
Notice, that this does not depend on the fact that $d$ is a length
metric.

To show the reverse inequality
%\begin{equation*}\label{reverse}
%u(p,t)\geq\min_{q\in X}\{(t-s)L(\frac{d(q,p)}{t-s})+Q_sg(q)\}
%\end{equation*}
we use the properties of the geodesic metric $d$. Again by
%Lemma\ref{lem: infminX}
(i) we can choose for $(p,t)\in X\times\R_+$ such $w\in X$ that
is minimizes \eqref{HLX}.% the expression in the Hopf-Lax formula,
%and
%\begin{equation*}
%Q_tg(p)=\min_{q\in
%X}\Big\{tL\bigg(\frac{d(q,p)}{t}\bigg)+g(q)\Big\}=tL\bigg(\frac{d(w,p)}{t}\bigg)+g(w)
%\end{equation*}
Now, if $q'\in X$ is on a length-minimizing path from $p$ to $w$,
we have
%With such a choice of $q'$ the triangle inequality is in fact an
%equality and we get
\begin{equation*}
d(w,p)=d(q',p)+d(w,q'),
\end{equation*}
and for a given $\sigma,\tau >0$ such that $\sigma+\tau = 1$ we
can find $q'\in X$ satisfying
\begin{equation*}
\begin{split}
d(q',p)= \tau d(w,p),\qquad d(w,q')= \sigma d(w,p).
\end{split}
\end{equation*}
By setting $\sigma=\frac{s}{t}$, and consequently
$\tau=\frac{t-s}{t}$, we obtain
\begin{equation*}
\frac{d(w,p)}{t}=\frac{t}{t-s}\frac{d(q',p)}{t}=\frac{t}{s}\frac{d(w,q')}{t}
\end{equation*}
and, moreover,
\begin{equation*}
L\bigg(\frac{d(w,p)}{t}\bigg)=L\bigg(\frac{d(q',p)}{t-s}\bigg)=L\bigg(\frac{d(w,q')}{s}\bigg).
\end{equation*}
This implies that
\begin{equation}
tL\bigg(\frac{d(w,p)}{t}\bigg)=(t-s)L\bigg(\frac{d(q',p)}{t-s}\bigg)
+ sL\bigg(\frac{d(w,q')}{s}\bigg). \label{lkj}
\end{equation}
Finally, we add $g(w)$ on both sides of (\ref{lkj}) and deduce
\begin{equation*}
\begin{split}
Q_tg(p) & =  tL\bigg(\frac{d(w,p)}{t}\bigg)+g(w)=(t-s)L\bigg(\frac{d(q',p)}{t-s}\bigg) + sL\bigg(\frac{d(w,q')}{s}\bigg)+g(w)\\
  & \geq  (t-s)L\bigg(\frac{d(q',p)}{t-s}\bigg)+\min_{v\in X}\Big\{sL\bigg(\frac{d(v,q')}{s}\bigg)+g(v)\Big\}\\
     & =  (t-s)L\bigg(\frac{d(q',p)}{t-s}\bigg)+Q_sg(q')\\
  & \geq  \min_{q\in X}\Big\{(t-s)L\left(\frac{d(q,p)}{t-s}\right)+Q_sg(q)\Big\}.
\end{split}
\end{equation*}
%This finishes the proof.
\end{proof}

%\begin{rem} \label{semigroup-s}
%Notice, that is clear from the proof of Proposition \ref{prop:
%semigroup} that for all $q\in X$ the inequality
%$$
%Q_tg(p) \leq (t-s)L\bigg(\frac{d(q,p)}{t-s}\bigg)+Q_sg(q)
%$$
%holds true without the assumption of geodesic metric.
% Properties
%of a geodesic space where used only for the the proof of the
%reverse inequality.

\begin{proof}[Proof of Theorem \ref{solutionHJ} (iii)]

By (ii), for a fixed $p\in X$ we have
\begin{equation}
\begin{split}
Q_tg(p) & =  \min_{q\in X}\Big\{(t-s)L\bigg(\frac{d(q,p)}{t-s}\bigg)+Q_sg(q)\Big\}\\
  & \leq  (t-s)L(0)+Q_sg(p)=Q_sg(p)\label{hilfsgleichung}
\end{split}
\end{equation}
by choosing $p=q$ and using $L(0)=0$.
%\end{rem}

\end{proof}

%The semi-group property is crucial in the theory of
%Hamilton-Jacobi equations. It will be used, for example, for
%proving the Lipschitz continuity of the Hopf-Lax formula.

%\begin{thm}\label{regularity}
%Let $(X,d)$ be proper and geodesic. Assume that  $g\colon
%X\rightarrow\R$ is Lipschitz continuous, and the Hamiltonian
%$H\colon\R_+\rightarrow\R_+$ satisfies $H(0)=0$, is increasing,
%convex and superlinear. Then the function $u$ given by (\ref{HLX})
%is Lipschitz continuous on $X\times\R_+$.

%\begin{thm}\label{regularity1}
% We assume that  $g:X\rightarrow\R$ is Lipschitz continuous, and the  Hamiltonian $H:\R_+\rightarrow\R_+$   satisfies $H(0)=0$, is increasing, convex and superlinear.
%Then the function $u$ given by (\ref{HLX}) is Lipschitz continuous on $X\times\R_+$, more precisely
%the inequality
%\begin{equation} \label{lipschitz}
%|Q_tg(p)-u(\xi,s)| \leq \max\{H(lip(g)), lip(g)\}(|t-s| +d(p,q))
%\end{equation}
% holds for all $(p,t), (\xi,s) \in X\times \R_+$.
%end{thm}

\begin{proof}[Proof of Theorem~\ref{solutionHJ} (iv).]
%In particular, we  have
In fact, we will prove that $$lip\left(Q_{(\cdot)} g(\cdot)\right)\leq \max\{lip(g), H(lip(g))\}$$
where $lip$ stands for the Lipschitz constant of the corresponding function (of one or two
variables). On $X\times \R_+$ we assume the canonical product metric
%\end{thm}
%Here we used the notation
%\[
%lip(g)= \sup_{x\neq y}\frac{| g(x)-g(y)|}{d(x,y)},
%\]
%for the Lipschitz constant of $g\colon X\to \R$. We define the
%Lipschitz constant of \mbox{$u\colon X\times \R_+ \to \R$} in the
%same way using the canonical product  metric
\[
d_{X\times\R_+}((x,t), (y,s))=  d(x,y) + | s-t|.
\]
We recall that by Rademacher's theorem the Lipschitz continuity of
$Q_{(\cdot)} g(\cdot)$ implies differentiability of $Q_{(\cdot)} g(x)$ a.e. in the $t$ variable.

We shall  fix $t>0$ and show the Lipschitz continuity of $x\to
Q_tg(x)$ first. %More precisely we prove
%\begin{equation} \label{lipschitz1}
%|Q_tg(x)-Q_tg(\xi)| \leq lip(g)d(p,\xi) \ \mbox{for all} \ p, \xi
%\in X.
%\end{equation}
Let $x,\xi\in X$ be arbitrary, and choose a minimizing $y_0$ in \eqref{HLX} for $(\xi,t)$. %such that
%$Q_tg(\xi)=tL\left(\frac{d(\xi,y_0)}{t}\right)+g(y_0)$.
By the
Lipschitz continuity of $g$ we get
\begin{equation}
\label{hilf}
\begin{split}
Q_tg(x)-Q_tg(\xi) &\leq  tL\bigg(\frac{d(q,x)}{t}\bigg)+g(q)-tL\bigg(\frac{d(\xi,y_0)}{t}\bigg)-g(y_0)\\
                       &\leq t\bigg[L\bigg(\frac{d(q,x)}{t}\bigg)-L\bigg(\frac{d(\xi,y_0)}{t}\bigg)\bigg]+lip(g)d(q,y_0)
\end{split}
\end{equation}
for any  $q\in X$.  %To make the appropriate choice for $q\in X$ we
%consider two cases.

Assume first that $d(x,y_0)\geq d(x,\xi)$. Choose $q$ on the
minimizing geodesic from $y_0$ to $x$ such that
$d(q,y_0)=d(x,\xi)$, and hence \mbox{$d(x,q)\leq d(\xi,y_0)$.}
%To
%see this write
%\begin{equation*}
%\begin{split}
%d(y_0,q)+d(q,x) &= d(y_0,x)\leq d(y_0,\xi)+d(\xi,x)\\
%               &= d(y_0,\xi)+d(q,y_0).
%\end{split}
%\end{equation*}
Since $L$ is increasing this with \eqref{hilf} implies that %we have
%$L\left(\frac{d(x,q)}{t}\right)\leq
%L\left(\frac{d(\xi,y_0)}{t}\right)$, and we get from
\begin{equation*}
\begin{split}
Q_tg(x)-Q_tg(\xi) &\leq t\bigg[L\bigg(\frac{d(q,x)}{t}\bigg)-L\bigg(\frac{d(\xi,y_0)}{t}\bigg)\bigg]+lip(g)d(q,y_0) \\
                             &\leq  lip(g)d(x,\xi).
\end{split}
\end{equation*}

Assume then that $d(x,y_0)<d(x,\xi)$. Since $Q_tg(x)\leq g(x)$,
choose $q=x$ in \eqref{hilf} to obtain
\begin{multline*}
%\begin{split}
Q_tg(x)-Q_tg(\xi) \leq g(x)-tL\bigg(\frac{d(\xi,y_0)}{t}\bigg)-g(y_0)\\
                             \leq lip(g)d(x,y_0)-tL\bigg(\frac{d(\xi,y_0)}{t}\bigg)\leq lip(g)d(x,\xi).
                            % &\leq lip(g)d(x,\xi).
%\end{split}
\end{multline*}
The two estimates now lead to
\[
Q_tg(x)-Q_tg(\xi) \leq lip(g)d(x,\xi)
\]
for all $x, \xi \in X,$ and simply interchanging $p$ and $\xi$
implies the desired Lipschitz continuity.
%$$Q_tg(\xi)-Q_tg(p) \leq lip(g)d(p,\xi)$$
%for all  $p, \xi \in X$. As a consequence we obtain the desired
%estimate \eqref{lipschitz1}.

%We will now use Lipschitz continuity of $Q_{(\cdot)} g(\cdot)$ in the space variable
%$x$ to prove the same property in the time variable $t$.
We now turn to the Lipschitz continuity of $t\to Q_tg(x)$. With no
loss of generality we assume $0<s<t$. Since $u$ is non--increasing
in $t$ we have $Q_tg(x)-Q_sg(x)\leq 0$. %To obtain a lower estimate
%for $Q_tg(p)-Q_sg(p)$, we shall use the the semi-group property.
By (ii) we get
\begin{equation*}
\begin{split}
Q_tg(x)      &= Q_sg(x)+\min_{q\in X}\Big\{(t-s)L\bigg(\frac{d(x,q)}{t-s}\bigg)+Q_sg(q)-Q_sg(x))\Big\}\\
     & \geq Q_sg(x)+\min_{q\in X}\Big\{(t-s)L\bigg(\frac{d(x,q)}{t-s}\bigg)-lip(g)d(x,q)\Big\}\\
        & \geq  Q_sg(x)+(t-s)\min_{v\in\R_+}\{L(v)-lip(g)v\} \\
        & =  Q_sg(x)-(t-s)H(lip(g)),
\end{split}
\end{equation*}
where $v=d(x,q)/(t-s)$. This shows that
\begin{equation*}%\label{lipschitz2}
|Q_tg(x)-Q_sg(x)|\leq H(lip(g))|t-s|.
\end{equation*}
%which is the Lipschitz continuity of $u$ in the $t$ variable.
Now the Lipschitz continuity in both variables imply
%\eqref{lipschitz1} and \eqref{lipschitz2} we obtain
\begin{equation*}
\begin{split}
|Q_tg(p)-Q_sg(\xi)| &\leq |Q_tg(p)-Q_sg(p)| +|Q_sg(p)-Q_sg(\xi)| \\
%&\leq  H(lip(g))|t-s| + lip(g) d(p,\xi)\\
&\leq \max\{H(lip(g)), lip(g)\}(|t-s| +d(p,\xi)).
\end{split}
\end{equation*}
%and the theorem follows.
\end{proof}

%\begin{thm}\label{ham-jac}
%In the setting of Theorem \ref{regularity} the function $Q_tg(x)$
%given by \eqref{HLX} solves the initial-value problem \eqref{HJEX}
%of the Hamilton-Jacobi equation for a.e. $t \in \R_+$.
%\end{thm}

%\begin{prop}\label{prop:geq}
%Using the notation and assumptions from Theorem \ref{ham-jac}, the
%inequality
%\begin{equation*}%\label{bigger-than-0}
%u_t(x,t)+ H(|\nabla u|(x,t)) \geq 0
%\end{equation*}
%holds for every $x\in X$ and a.e. $t\in \R_+$.
%\end{prop}

%\begin{prop}\label{prop:leq}
%In the setting of Theorem \ref{regularity}, the inequality
%\begin{equation}\label{less-than-0}
%u_t(x,t)+ H(|\nabla^- u|(x,t)) \leq 0
%\end{equation}
%holds for every $x\in X$ and a.e. $t\in \R_+$.
%\end{prop}

\begin{proof}[Proof of Theorem \ref{solutionHJ} (v).]

%The following two propositions prove Theorem \ref{ham-jac}. We
%skip the proof of Proposition~\ref{prop:geq} since it will follow
%from Proposition~\ref{prop:liminf}.
We show that
\begin{equation}\label{less-than-0}
\frac{\partial}{\partial t}u(x,t)+ H(|\nabla^- u|(x,t)) \leq 0
\end{equation}
holds for every $x\in X$ and a.e. $t\in \R_+$ for
$u(x,t)=Q_tg(x)$. The converse inequality follows from (vi).
%Theorem~\ref{solutionHJ} (vi).

Fix $x\in X$ and let $t\in \R_+$ a point of differentiability of
$u(x,\cdot)$. If $|\nabla^- u|(x,t) =0$, \eqref{less-than-0}
reduces to $u_t(x,t) \leq 0$ since $H(0)=0$. This clearly holds
since $u(x,\cdot)$ is non--increasing.

We can thus assume that $|\nabla^- u|(x,t) >0$, and there exists a
sequence $x_n \to x$ for which $u(x_n,t) < u(x,t)$ and
$$|\nabla^- u|(x,t) = \lim_{n\to \infty} \frac{u(x,t)-u(x_n,t)}{d(x_n,x)}.$$
For the moment, consider any positive sequence $(h_n)$ with
$h_n\to 0$. By the semi--group property (ii) we get
\[
u(x, t+h_n) = \min_{y\in
X}\Big\{h_nL\bigg(\frac{d(x,y)}{h_n}\bigg) + u(y,t)\Big\} \leq
h_nL\bigg(\frac{d(x,x_n)}{h_n}\bigg) + u(x_n,t),
\]
which implies that
\begin{equation}\label{upper est}
\begin{split}
\frac{u(x,t+h_n)-u(x,t)}{h_n} &\leq  %L\bigg(\frac{d(x,x_n)}{h_n}\bigg) + \frac{u(x_n,t)-u(x,t)}{h_n} \\
%&=
-\left[\frac{u(x,t)-u(x_n,t)}{h_n}-L\bigg(\frac{d(x,x_n)}{h_n}\bigg)\right].
\end{split}
\end{equation}
Since $H(w)= \max_{v\in \R_+} \{wv-L(v)\} $ for all $w\in \R_+$,
for each $n$ it is possible to choose $h_n>0$ such that
\begin{equation}\label{hn}
H\bigg(\frac{u(x,t)-u(x_n,t)}{d(x_n,x)}\bigg) =
\frac{u(x,t)-u(x_n,t)}{h_n}-L\bigg(\frac{d(x,x_n)}{h_n}\bigg)
\end{equation}
holds. Furthermore, it is easy to see directly from \eqref{hn}
that $x_n\to x$ implies $h_n \to 0$.

Finally, combining \eqref{upper est} and \eqref{hn} we obtain
$$ \frac{u(x,t+h_n)-u(x,t)}{h_n} + H\left(\frac{u(x,t)-u(x_n,t)}{d(x_n,x)}\right) \leq 0.$$
As $x_n\to x$ and $h_n\to 0$, letting $n\to\infty$ gives us
\eqref{less-than-0}.
\end{proof}

%\begin{thm}\label{ham-jac2}
%If, in addition to assumption in Theorem \ref{regularity},
%$(X,d,\mu)$ supports a local $(1,p)$--Poincar\'e inequality and
%$\mu$ is doubling, then the function $Q_tg(x)$ given by \eqref{HLX}
%solves the initial-value problem \eqref{HJEX} of the
%Hamilton-Jacobi equation for all $t \in \R_+$ and $\mu$--a.e.
%$x\in X$.
%\end{thm}
%\marginpar{L\&V assume $(1,1)$-Poincare even if it seems that
%$(1,p)$ is enough}

%The following two propositions prove Theorem \ref{ham-jac2}.
%We will now prove Theorem \ref{ham-jac}. The first result in this
%direction is one of the inequalities which holds even in a
%stronger version than as required by the theorem. The statement
%is formulated as
%\begin{prop}\label{prop:liminf}
%In the setting of Theorem \ref{regularity} the inequality
%\begin{equation}%\label{lim-bigger-than-0}
%\liminf_{s\to 0^{+}}\frac{u(x,t+s)-Q_tg(x)}{s}+ H\big(|\nabla^-
%u|(x,t)\big) \geq 0
%\end{equation}
%holds for every $x\in X$ and $t\in \R_+$.
%\end{prop}

%\begin{rem}
%1. Notice, that by the fact that $|\nabla  Q_tg(x)| \geq |\nabla^- u|(x,t)$ and since $H$ is non-decreasing    our statement yields
%\begin{equation*}\label{bigger-than-02}
%u_t(x,t)+ H(|\nabla^- u|(x,t)) \geq 0,
%\end{equation*}
%which is in fact needed for the proof of Theorem \ref{ham-jac}.
%
%Observe that here the full measure set of $t\in \R_+$ for which
%\eqref{bigger-than-0} holds depends on the point $x\in X$. This
%et is in fact that set points of differentiability for the
%function $t\to Q_tg(x)$ for fixed $x$.
\begin{proof}[Proof of Theorem \ref{solutionHJ} (vi)]

Let us fix $x\in X$ and $t\in \R_+$. Since $(x,t)\mapsto Q_tg(x)$
is a Lipschitz function, the limes inferior in \eqref{liminf}
% in \eqref{lim-bigger-than-0}
is finite and we can choose a positive sequence $(h_n)$ such that
$h_n\to 0$ and
\begin{equation}
\label{realizing_infimum} \liminf_{s\to
0^{+}}\frac{Q_{t+s}g(x)-Q_tg(x)}{s} = \lim_{n\to
\infty}\frac{Q_{t+h_{n}}g(x)-Q_tg(x)}{h_n}.
\end{equation}

Next, applying the semigroup property we can write
\begin{equation}
\label{semihn} Q_{t+h_{n}}g(x)= \min_{y\in X }\Big\{
h_nL\bigg(\frac{d(x,y)}{h_n} \bigg)+Q_tg(y)\Big\}.
\end{equation}
For each $n$ we choose a point $y_n\in X$ for which the minimum is
attained. The superlinearity of $L$ implies that $y_n\to x$.
% and an argument similar to the one in the proof of Lemma
%\ref{lem: infminX}.

As $Q_tg(x)$ is decreasing in $t$, we have $Q_{t+h_{n}}g(x)\leq
Q_tg(x)$, and hence
\begin{equation}
\label{positive} Q_tg(y_n)\leq h_nL\bigg(\frac{d(x,y)}{h_n}
\bigg)+Q_tg(y_n) \leq Q_tg(x).
\end{equation}
Since $H(w) = \max _{v\in \R_+}\{ w v -L(v)\}$ we have $H(w)+L(v)
\geq wv$ for all $w,v\in \R_+$. Together with \eqref{positive}
this implies that
$$H\bigg(\frac{Q_tg(x)-Q_tg(y_n)}{d(x,y_n)}\bigg) + L\bigg(\frac{d(x,y_n)}{h_n}\bigg)\geq \frac{Q_tg(x)-Q_tg(y_n)}{h_n},$$
and we have
\begin{equation*}%\label{L-H}
L\bigg(\frac{d(x,y_n)}{h_n}\bigg) +
\frac{Q_tg(y_n)-Q_tg(x)}{h_n}\geq
-H\bigg(\frac{[Q_tg(x)-Q_tg(y_n)]_+}{d(x,y_n)}\bigg).
\end{equation*}
Together with \eqref{semihn} this implies
\begin{equation*}% \label{tnhn2}
\begin{split}
\frac{Q_{t+h_{n}}g(x)-Q_tg(x)}{h_n}& = \frac{1}{h_n} \bigg(h_n
L\bigg(\frac{d(x,y_n)}{h_n}\bigg) +Q_tg(y_n)-Q_tg(x)\bigg)
\\
%& = L(\frac{d(x,y_n)}{h_n}) +\frac{Q_tg(y_n)-Q_tg(x)}{h_n}\\
& \geq  %-H\bigg(\frac{Q_tg(x)-Q_tg(y_n)}{d(x,y_n)}\bigg)=
-H\bigg(\frac{[Q_tg(x)-Q_tg(y_n)]_+}{d(x,y_n)}\bigg).
\end{split}
\end{equation*}
Letting now $n\to \infty$ and using \eqref{realizing_infimum} we
obtain
\begin{equation*}
\begin{split}
\liminf_{s\to 0^{+}}\frac{Q_{t+s}g(x)-Q_tg(x)}{s}
%&=\lim_{n\to\infty}\frac{u(x,t+h_n)-Q_tg(x)}{h_n}\\
%&=\limsup_{n\to\infty}\frac{u(x,t+h_n)-Q_tg(x)}{h_n}\\
&\geq \limsup_{n\to\infty}\bigg(-
H\bigg(\frac{[Q_tg(x)-Q_tg(y_n)]_+}{d(x,y_n)}\bigg)\bigg) \\
%&=-\liminf_{n\to\infty}\bigg(
%H\bigg(\frac{[Q_tg(x)-Q_tg(y_n)]_+}{d(x,y_n)}\bigg)\bigg)\\
% &\geq -\limsup_{n\to\infty}\bigg(-
%H\bigg(\frac{[Q_tg(x)-Q_tg(y_n)]_+}{d(x,y_n)}\bigg)\bigg)\\
% // &\geq  -\limsup_{y\to x} H\bigg(\frac{\big(Q_tg(x)-Q_tg(y_n)\big)_+}{d(x,y_n)}\bigg)
&\geq -H\big(|\nabla^- Q_tg|(x)\big).
\end{split}
\end{equation*}
%finishing the proof of the proposition.
\end{proof}

Notice, that if $u(x,t)=Q_tg(x)$, and $t$ is a point of
differentiability of $t\to u(x,t)$ for a fixed $x$, then it
follows from (vi) that
\begin{equation*}%\label{bigger-than-0}
u_t(x,t)+ H\big(|\nabla^- u|(x,t)\big) \geq 0.
\end{equation*}
Since $u$ is Lipschitz--continuous, the above inequality holds for
all $x\in X$ and a.e. $t\in \R_+$. This finishes the proof of (v).

%\begin{prop}\label{limsupprop} If, in addition to assumptions in
%Theorem~\ref{regularity}, $(X,d,\mu)$ supports a local Poincar\'e
%inequality, and $\mu$ is doubling then
%\begin{equation}%\label{limsup}
%\limsup_{s\to 0^+} \frac { u(x,t+s) - Q_tg(x) }s \le  - H\big( |
%\nabla^- u|(x,t)\big)
%\end{equation}
%for $t>0$ and a.e. $x\in X$.
%\end{prop}

\begin{proof}[Proof of Theorem \ref{solutionHJ} (vii)]

We prove (vii) along the lines in~\cite[Thm 2.5 (vii)]{LoVi}. If
$|\nabla^- Q_tg|(x) =0$ the statement is trivial since $Q_tg(x)$
is non-increasing in $t$. Let $t>0$ be fixed and assume that $|\nabla^-
Q_tg|(x)> 0$. Define $f(x):= Q_tg(x)$ and fix a real number
$\alpha>0$. By the semi-group property (ii) we get for $s>0$
\begin{equation*}
\begin{split}
\frac{ Q_tg(x) - Q_{t+s}g(x) } s  & = \frac 1s \sup_{y\in X}
\left[f(x) -f(y) - s \, L\left( \frac {d(x,y)} s\right)  \right]
\\ &
\ge  \sup_{y\in S_{\alpha s }(x)}\left[ \frac
{f(x)-f(y)}{d(x,y)}\: \alpha- L(\alpha)\right].
\end{split}
\end{equation*}
Write $$\psi(r) = \sup_{y\in S_r(x)}\frac {f(x)  - f(y)}
{d(x,y)}.$$ It is shown in~\cite{LoVi} that $\liminf_{r\to 0^+}
\psi(r)= |\nabla^- f|(x)$ a.e. on $X$. Thus
\begin{equation*}
\liminf_{s\to 0^+}\frac{ Q_tg(x) - Q_{t+s}g(x) } s \ge |\nabla^-
Q_tg|(x)\;\alpha - L(\alpha).
\end{equation*}
Maximizing the above inequality over $\alpha > 0$ we obtain that
\begin{equation*}
\liminf_{s\to 0^+}\frac{ Q_tg(x) - Q_{t+s}g(x) } s \ge
H\left(|\nabla^- Q_tg|(x)\right),
\end{equation*}
which is equivalent to the statement of the proposition.
\end{proof}

Finally, (vi) and (vii) together prove (viii).

\section{\bf Applications, comments and questions}
A large class of geodesic metric measure spaces for which the Poincar\'e inequality holds --  and our results apply -- are the Carnot-Carath\'eodory geometries; see, for example, \cite{HeKo} and \cite{Gr}. A case of particular interest within this class is the class of Carnot groups where many fundamental results of Euclidean analysis hold.  In this setting, Hamilton--Jacobi equations have already been considered  by Manfredi and Stroffolini  \cite{MaSt}, see also  \cite{Dr}. It would be interesting to characterize measures  for which an appropriate Log-Sobolev inequality holds on Carnot-Carath\'eodory spaces. In the Euclidean setting results in this direction were obtained by Barthe and Kolesnikov  \cite{BaKo}. In the case of the first Heisenberg  group $\mathbb{H}$, Inglis and Papageorgiou showed in the recent paper ~\cite{InPa} that the measure 
$$ \mu_p (dx) = \frac{e^{-\beta d^p(x)}}{\int_{\mathbb{H}} e^{-\beta d^p(x)}dx}dx$$
satisfies the $q$-Log-Sobolev inequality. Here $\beta > 0$ is an arbitrary number, $p\geq 2$ is the conjugate exponent to $q$, $dx$ is the Lebesgue measure and $d(x)$ is the sub-Riemannian 
Carnot-Carath\'eodory distance on $\mathbb{H}$. In order to apply our to apply our results, 
one has to note that for smooth functions $f:\mathbb{H} \to \R$ the norm of the sub-Riemannian gradient $|\nabla f(x)|$ from \cite{InPa} and our 
metric subgradient $|\nabla^- f(x)|$ coincide for $\mu_p$ a.e. $x$ for which $|\nabla f(x)|>0$. For Lipschitz continuous functions this  follows from Pansu's differentiability theorem (\cite{Gr}). 

Therefore the $q$-Log-Sobolev inequality 
according to Definition \ref{logsob} holds in this setting. Applying our results one obtains the validity of the $p$-Talagrand inequality and hypercontractivity of the Hamilton-Jacobi semigroup in the setting of the Heisenberg group equipped with the sub-Riemannian metric and the above probability measure 
$\mu_p$. 

Furthermore, it would be interesting to see whether the results of this paper hold in the more general class of metric measure spaces satisfying the so--called Lip--lip condition. To be precise, let us recall from \cite{Ke2} that a metric measure space $(X, d, \mu)$ satisfies the Lip--lip condition if there exists a constant $L \geq 1$ with the property that if $f:X\to \R$ is a Lipschitz function then 
\begin{equation} \label{Lip-lip}
Lip f (x) \leq L \cdot lip f(x), \ \ \mbox{for} \ \mu\mbox{--a.e.} \ x\in X,  
\end{equation} 
where $Lip f(x)$ and $lip f (x)$ are the local Lipschitz numbers of $f$ at $x$ defined as
$$Lip f (x) = \limsup_{r\to 0} \sup_{y\in B(x,r)} \frac{|f(x)-f(y)|}{r} , $$
$$lip f (x) = \liminf_{r\to 0} \sup_{y\in B(x,r)} \frac{|f(x)-f(y)|}{r} . $$
Let us recall that Keith proved in \cite{Ke2} that if a metric measure space $(X,d,\mu)$, where $\mu$ is doubling, satisfies the Lip--lip condition then $X$ supports a measurable differentiable structure in the sense of Cheeger 
\cite{Che}. Keith also proved that if the doubling metric measure space $(X,d,\mu)$  satisfies the Poincar\'e inequality then the Lip--lip condition is satisfied. It is also clear that the Lip--lip condition is more general than the Poincar\'e inequality, for example a positive measure Cantor set in the Euclidean space satisfies this condition but does not support a Poincar\'e inequality. 

Finally, it would be interesting to prove a variant of Hopf-Lax formula for the solution of the Hamilton-Jacobi equation, i.e. Theorem \ref{solutionHJ} for the case of geodesic spaces satisfying the Lip-lip condition. It is clear that statements (i) through (vi) will hold true without modification. Furthermore, it is reasonable to expect that statement (vii) will be replaced by 

\begin{equation*}
\limsup_{s\to 0^+} \frac { Q_{t+s}g(x) - Q_tg(x) }s \le  - H\big(\frac{|
\nabla^- Q_tg|(x)}{L'}\big)
\end{equation*}
for all $t>0$ and $\mu$--a.e. $x\in X$ and for some absolute constant  $L'\geq 1$ depending on $(X,d,\mu)$.  

The statements  the other results of the paper concerning the circle of equivalences of Talagrand Log- Sobolev inequality and hypercontractivity would then follow (with possibly adjusted constants) along the same lines as in the case of metric spaces satisfying a Poincar\'e inequality.  
\medskip 

\noindent {\bf Acknowledgements} The authors are grateful to Cedric Villani for his valuable comments and observations concerning the manuscript.

\bibliographystyle{plain}
%\bibliography{references}
%\nocite{*}

\end{document}